\documentclass[11pt]{amsart}
\usepackage[T1]{fontenc}
\usepackage{geometry}
\usepackage{graphicx}
\usepackage{amssymb}

\makeatletter

\providecommand{\tabularnewline}{\\}

 \theoremstyle{plain}    
 \newtheorem{thm}{Theorem}[section]
 \numberwithin{equation}{section} 
 \numberwithin{figure}{section} 
 \theoremstyle{plain}
 \theoremstyle{definition}
 \newtheorem{defn}[thm]{Definition}
 \theoremstyle{plain}    
 \newtheorem{prop}[thm]{Proposition} 
 \theoremstyle{plain}    
 \newtheorem{cor}[thm]{Corollary} 
 \theoremstyle{definition}
  \newtheorem{example}[thm]{Example}

\makeatother
\begin{document}

\title{Rectangular Polyomino Set Weak (1,2)-achievement Games}

\author{Edgar Fisher}

\author{N\'andor Sieben}

\begin{abstract}
In a polyomino set (1,2)-achievement game the maker and the breaker
alternately mark one and two previously unmarked cells respectively.
The maker's goal is to mark a set of cells congruent to one of a given
set of polyominoes. The breaker tries to prevent the maker from achieving
his goal. The teams of polyominoes for which the maker has a winning
strategy is determined up to size 4. In set achievement games, it
is natural to study infinitely large polyominoes. This enables the
construction of super winners that characterize all winning teams
up to a certain size. 
\end{abstract}

\address{Northern Arizona University, Department of Mathematics and Statistics}

\email{edgar.fisher@nau.edu}

\email{nandor.sieben@nau.edu}



\maketitle

\section{Introduction}

A \emph{rectangular board} is the set of \emph{cells} that are the
translations of the unit square $[0,1]\times[0,1]$ by vectors of
$\mathbf{Z}^{2}$. Informally, a rectangular board is the infinite
chessboard. Two cells are called \emph{adjacent} if they share a common
edge. A \emph{polyomino} (or animal) is a subset of the rectangular
board in which the cells are connected through adjacent cells. Note
that we allow infinitely many cells in a polyomino. We only consider
polyominoes up to congruence, that is, the location of the polyomino
on the board is not important. Rotations and reflections are also
allowed. The number of cells of a polyomino is called the \emph{size}
of the polyomino.

In a \emph{polyomino set $(p,q)$-achievement game} two players alternately
mark $p$ and $q$ previously unmarked cells of the board using their
own colors. If $p$ or $q$ is not 1 then the game is often called
\emph{biased}. In a regular game, the player who first marks a polyomino
congruent to one of a given set of finite polyominoes wins the game.
In a \emph{weak set achievement} game the second player (\emph{the
breaker}) only tries to prevent the first player (\emph{the maker})
from achieving one of the polyominoes. A set of finite polyominoes
is called a \emph{winning set} if the maker has a winning strategy
to achieve this set. Otherwise the set is called a \emph{losing set}.
Polyomino achievement games were introduced by Harary \cite{gardner.martin:mathematical,harary:achieving,harary:is,harary.harborth.ea:handicap}.
Winning strategies on rectangular boards can be found in \cite{bode.harborth:hexagonal,sieben.deabay:polyomino}.
Biased games are studied in \cite{beck:biased} in a more general
setting. Biased games are needed \cite{sieben:snaky} to apply the
theory of weight functions \cite{beck:remarks,erdos.selfridge:on}
to unbiased games on infinite boards.

In this paper we study rectangular weak set $(1,2)$-achievement games.
Triangular unbiased set achievement games were studied in \cite{bode.harborth:triangle}.
Our purpose is to further develop the theory of set achievement games.
We have chosen the rectangular game because the rectangular board
is the most intuitive. The unbiased rectangular set game is very complex.
To handle this difficulty we have chosen a biased version to limit
the number of winning sets. The $(1,2)$ game is still rich enough
to uncover many of the unexpected properties of set games. This approach
also has its challenges, since the $(1,2)$ game needs new tools for
finding winning strategies.

\section{Preliminaries}

\begin{figure}[h]
\begin{tabular}{ccccccccccccccc}
\includegraphics{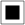}&
&
\includegraphics{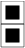}&
&
\includegraphics{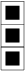}&
\includegraphics{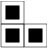}&
&
\includegraphics{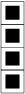}&
\includegraphics{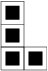}&
\includegraphics{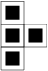}&
\includegraphics{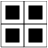}&
\includegraphics{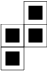}&
&
&
\includegraphics{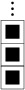}\tabularnewline
$P_{1,1}$&
&
$P_{2,1}$&
&
$P_{3,1}$&
$P_{3,2}$&
&
$P_{4,1}$&
$P_{4,2}$&
$P_{4,3}$&
$P_{4,4}$&
$P_{4,5}$&
&
&
\tabularnewline
$S_{1}$&
&
$S_{2}$&
&
$S_{3}$&
$L_{2}$&
&
$S_{4}$&
$L_{3}$&
$T_{2}$&
$C_{2}$&
$Z_{2}$&
&
&
$S_{\infty}$\tabularnewline
\end{tabular}

\caption{\label{cap:Polyominoes-up-to}All polyominoes up to size 4 together
with infinite skinny.}
\end{figure}

Figure~\ref{cap:Polyominoes-up-to} shows some polyominoes we are
going to use. In this figure, the polyominoes are in standard position.
Roughly speaking, a polyomino is in standard position if its cells
are as much to the left and to the bottom as possible. The exact definition
involves the lexicographic order of the list of coordinates of the
cells of the polyomino pushed against the coordinate axes in the first
quadrant. The naming convention comes from the ordering of the polyominoes
by size and by lexicographic order of their standard position. 

We use special names for several important classes of polyominoes.
These names are also given in the figure. The name $S_{n}=P_{n,1}$
stands for the skinny polyomino of size $n$. The names $C_{n}$,
$L_{n}$, $T_{n}$ and $Z_{n}$ are chosen because the shape of those
polyominoes is similar to the shape of letters. Note that only one
end of $S_{\infty}$ is infinitely long. 

\begin{defn}
A set of polyominoes is called \emph{bounded} if it contains only
finite polyominoes. It is called \emph{unbounded} if it contains at
least one infinite polyomino.
\end{defn}
Note that an infinite set of finite polyominoes is still called bounded
even though the size of a polyomino in the set can be arbitrarily
large.

\begin{defn}
We say the polyomino $P$ is an \emph{ancestor} of the polyomino $Q$
if $Q$ can be constructed from $P$ by adding some (possibly none)
extra cells. We use the notation $P\sqsubseteq Q$. A set $\mathcal{F}$
of polyominoes is called a \emph{team} if no element of $\mathcal{F}$
is the ancestor of another element of $\mathcal{F}$.
\end{defn}
It is easy to see that the ancestor relation is reflexive and transitive.
It is not antisymmetric, the polyominoes in Figure \ref{cap:Two-polyominoes-which}
are ancestors of each other. The relation is antisymmetric on finite
polyominoes and so is a partial order on the set of finite polyominoes.

\begin{figure}
\includegraphics{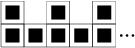}$\qquad$\includegraphics{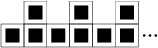}

\caption{\label{cap:Two-polyominoes-which}Two polyominoes which are ancestors
of each other.}
\end{figure}

So far we have not defined the term winner for an unbounded set of
polyominoes. An infinite polyomino cannot be marked during a finite
game. We still want to talk about unbounded winners to simplify the
theory, even though we do not intend to play any games with unbounded
sets.

\begin{defn}
Let $\mathcal{T}$ be an unbounded set of polyominoes. Let $F_{T}$
be a finite ancestor of $T$ for all $T\in\mathcal{T}$. Then $\mathcal{F}=\{ F_{T}\mid T\in\mathcal{T}\}$
is called a \emph{bounded restriction} of $\mathcal{T}$. An unbounded
set of polyominoes is called a \emph{winner} if each bounded restriction
of the set is a winner.
\end{defn}

\section{Preorder}

There are two ways to make it easier to achieve a set of polyominoes.
We can make some of the polyominoes smaller or we can include more
polyominoes in the set. This motivates the following definition. 

\begin{defn}
Let $\mathcal{S}$ and $\mathcal{T}$ be sets of polyominoes. We say
$\mathcal{S}$ is \emph{simpler} than $\mathcal{T}$ if for all $Q\in\mathcal{T}$
there is a $P\in\mathcal{S}$ such that $P\sqsubseteq Q$. We use
the notation $\mathcal{S}\preceq T$. 
\end{defn}
The terminology \emph{at least} and \emph{at most} was used in \cite{bode.harborth:triangle}
for what we call simpler. Note that $\mathcal{S}$ is simpler then
$\mathcal{T}$ if $\mathcal{S}$ is simpler to achieve than $\mathcal{T}$.
It is easy to see that the simpler relation is reflexive and transitive
and so is a preorder. It is also easy to see that a bounded restriction
of an unbounded set of polyominoes is simpler than the original set.
The following result shows the importance of the preorder.

\begin{prop}
\label{pro:order}Let $\mathcal{S}$ and $\mathcal{T}$ be sets of
polyominoes such that $\mathcal{S}\preceq\mathcal{T}$. If $\mathcal{T}$
is a winner then so is $\mathcal{S}$. If $\mathcal{S}$ is a loser
then so is $\mathcal{T}$.
\end{prop}
\begin{proof}
First assume that $\mathcal{S}$ and $\mathcal{T}$ are bounded. If
$\mathcal{T}$ is a winner then during a game the maker is able to
mark the cells of some $Q\in\mathcal{T}$. There is a $P\in\mathcal{S}$
such that $P\sqsubseteq Q$, so by the time the maker marks the cells
of $Q$ he also marked the cells of $P$, possibly at an earlier stage. 

Next assume that $\mathcal{S}$ is bounded and $\mathcal{T}$ is unbounded.
For each $T\in\mathcal{T}$ define $F_{T}=T$ if $T$ is finite and
define $F_{T}$ to be an element of $\mathcal{S}$ such that $F_{T}\sqsubseteq T$
if $T$ is infinite. Then $\mathcal{F}=\{ F_{T}\mid T\in\mathcal{T}\}$
is a bounded restriction of $\mathcal{T}$. $\mathcal{S}$ is simpler
than $\mathcal{F}$ and $\mathcal{F}$ is a winner and so $\mathcal{S}$
is also a winner.

Finally assume that $\mathcal{S}$ is unbounded. Let $\mathcal{E}$
be a bounded restriction of $\mathcal{S}$. Then $\mathcal{E}\preceq\mathcal{S}\preceq\mathcal{T}$
and so $\mathcal{E}$ is a winner which implies that $\mathcal{S}$
is a winner.

The second statement of the proposition is the contrapositive of the
first statement.
\end{proof}
\begin{defn}
Let $\mathcal{S}$ be a bounded set of polyominoes. The set $\mathcal{L}(\mathcal{S})$
of minimal elements of $\mathcal{S}$ in the partial order is called
the \emph{legalization} of $\mathcal{S}$.
\end{defn}
It is clear that $\mathcal{L}(\mathcal{S})$ is a team. 

\begin{prop}
\label{pro:legal}Let $\mathcal{S}$ be a bounded set of polyominoes.
$\mathcal{S}$ is a winner if and only if $\mathcal{L}(\mathcal{S})$
is a winner.
\end{prop}
\begin{proof}
Since $\mathcal{L}(\mathcal{S})$ is a subset of $\mathcal{S}$, we
must have $\mathcal{S}\preceq\mathcal{L}(\mathcal{S})$. On the other
hand, consider $Q\in\mathcal{S}$. If $Q$ is minimal then $Q\in\mathcal{L}(\mathcal{S})$.
If $Q$ is not minimal then there is a minimal $R\in\mathcal{S}$
such that $R\sqsubseteq Q$ and so $R\in\mathcal{L}(\mathcal{S})$.
This shows that $\mathcal{S}\succeq\mathcal{L}(\mathcal{S})$. The
result now follows from Proposition~\ref{pro:order}.
\end{proof}
Note that the existence of the minimal $R$ in the proof is not guaranteed
if $\mathcal{S}$ is unbounded. There could be an infinite chain $Q_{1}\sqsupseteq Q_{2}\sqsupseteq\cdots$
of simpler and simpler polyominoes without a minimal polyomino. This
means that we cannot talk about the legalization of an unbounded set
of polyominoes.

Proposition~\ref{pro:legal} allows us to concentrate on teams instead
of sets of polyominoes in order to classify sets of finite polyominoes
as winners or losers.

\section{Winning teams}

The \emph{exterior perimeter} of a polyomino is the number of empty
cells adjacent to the polyomino. The minimum exterior perimeter of
the polyominoes in a finite set $\mathcal{F}$ is denoted by $\varepsilon(\mathcal{F})$.
The \emph{full team} $\mathcal{F}_{s}$ is the set containing all
polyominoes of size $s$.

\begin{prop}
\label{pro:The-full-family}The full team $\mathcal{F}_{s}$ is a
winner for $s\le4$. In fact the maker can win after $s$ marks. 
\end{prop}
\begin{proof}
The maker can win after $s$ marks with the \emph{random neighbor
strategy} \cite{sieben:polyominoes}, which requires him to place
his mark at a randomly chosen cell adjacent to one of his previous
marks. The strategy works because $\varepsilon(\mathcal{F}_{1})=4$,
$\varepsilon(\mathcal{F}_{2})=6$, $\varepsilon(\mathcal{F}_{3})=7$
and $\varepsilon(\mathcal{F}_{4})=8$ and so $\varepsilon(\mathcal{F}_{s})$
is not larger than the number of cells marked by the breaker, which
is $2s$ after $s$ moves.
\end{proof}
It is not hard to see that $\mathcal{F}_{4}$ remains a winner if
we replace $\mathcal{S}_{4}$ by a larger skinny polyomino.

\begin{prop}
\label{pro:W_n}The team $\mathcal{W}_{n}=\{ S_{n+1},T_{2},C_{2},\ldots,C_{n},Z_{2},\ldots,Z_{n}\}$
with the polyominoes in Figure~\ref{cap:Infinite-winner} is a winner
for all $n\ge3$.
\end{prop}
\begin{figure}
\begin{tabular}{cccccccc}
\includegraphics{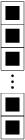}&
&
&
&
\includegraphics{4-4.eps}&
\includegraphics{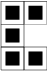}&
$\cdots$&
\includegraphics{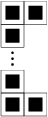}\tabularnewline
$S_{n+1}$&
&
&
&
$C_{2}$&
$C_{3}$&
&
$C_{n}$\tabularnewline
&
&
&
&
&
&
&
\tabularnewline
\includegraphics{4-3.eps}&
&
&
&
\includegraphics{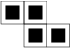}&
\includegraphics{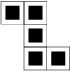}&
$\cdots$&
\includegraphics{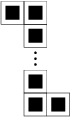}\tabularnewline
$T_{2}$&
&
&
&
$Z_{2}$&
$Z_{3}$&
&
$Z_{n}$\tabularnewline
\end{tabular}

\caption{\label{cap:Infinite-winner}The winner $\mathcal{W}_{n}$.}
\end{figure}

\begin{proof}
The maker can mark one of the polyominoes in $\mathcal{F}_{4}=\{ S_{4},L_{3},T_{2},C_{2},Z_{2}\}$
after four marks by Proposition~\ref{pro:The-full-family}. If this
polyomino is $T_{2}$, $C_{2}$ or $Z_{2}$ then the maker achieved
$\mathcal{W}_{n}$ and we are done. 

First consider the case when the marked polyomino is $S_{4}$. We
show by induction that even in this case the maker is able to achieve
$S_{n+1}$ and win or achieve $L_{k}$ for some $4\le k\le n$. Consider
Figure~\ref{cap:cases}(a) that shows the situation before the fifth
move of the maker. If the breaker has no marks in the cells containing
the letter A, then the maker can mark one of those cells and achieve
$T_{2}$. If the breaker has no marks in the cells containing the
letter B then the maker can mark one of those cells and achieve $L_{4}$.
So we can assume that the eight marks of the breaker are the cells
with the letters A and B. This completes the base step of the induction.
Now assume that we are in the situation shown in Figure~\ref{cap:cases}(b)
where the the maker already marked $S_{j-1}$ and the small empty
squares show the marks of the breaker. The maker now can mark the
cell containing the letter A. If the breaker does not answer by marking
the two cells containing the letter B then the maker can mark one
of these cells and achieve $L_{j}$. On the other hand if the breaker
marks these two cells then we are again in the situation shown in
Figure~\ref{cap:cases}(b) but the size of the polyomino $S_{j}$
marked by the maker is increased by one. Hence the maker eventually
achieves $S_{n+1}$ or $L_{k}$.

\begin{figure}
\begin{tabular}{ccccccc}
\includegraphics{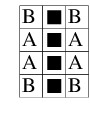}&
&
\includegraphics{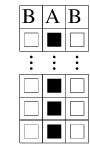}&
&
\includegraphics{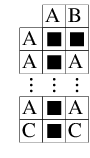}&
&
\includegraphics{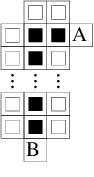}\tabularnewline
(a)&
&
(b)&
&
(c)&
&
(d)\tabularnewline
\end{tabular}

\caption{\label{cap:cases}Situations to achieve $\mathcal{W}_{n}$.}
\end{figure}

It suffices to consider the situation shown in Figure~\ref{cap:cases}(c)
where the maker marked $L_{k}$ after $k+1$ marks. If the breaker
has no marks in the cells containing the letter A, then the maker
can mark one of those cells and achieve $T_{2}$. If the breaker has
no mark in the cell containing the letter B, then the maker can mark
that cell and achieve $Z_{2}$. If the breaker has no marks in the
cells containing the letter C, then the maker can mark one of those
cells and achieve $C_{k}$ or $Z_{k}$. So we can assume that we are
in the situation shown in Figure~\ref{cap:cases}(d). Note that the
breaker can have $2k+2$ marks on the board while only $2k+1$ of
those marks are shown as forced moves. Without this extra mark, the
maker would have two ways to finish the game. He could mark the cell
containing the letter A and mark cells to the right of his previous
mark until he can make a turn up or down. He could also mark the cell
containing the letter B and mark cells below his previous mark until
he make a turn left or right. An inductive argument similar to the
one above shows that either way he can achieve $S_{n+1}$ without
a turn or he can achieve $C_{j}$ or $Z_{j}$ for some $3\le j\le n$.
The one extra mark of the breaker cannot ruin both of these ways to
win since the cells involved are disjoint.
\end{proof}
\begin{cor}
\label{cor:The-transfinite-family}The unbounded team \[
\mathcal{W}=\{ S_{\infty},T_{2}\}\cup\{ C_{n}\mid n\ge2\}\cup\{ Z_{n}\mid n\ge2\}\]
 is a winner.
\end{cor}
\begin{proof}
The bounded restrictions of $\mathcal{W}$ are all simpler than $\mathcal{W}_{n}$
for some $n$.
\end{proof}
\begin{cor}
\label{cor:Small_winners}The teams $\{ P_{2,1}\}$, $\{ P_{n,1},P_{3,2}\}$
for $n\ge3$ and $\{ P_{3,1},P_{4,4},P_{4,5}\}$ are winners.
\end{cor}
\begin{proof}
The first and the third team is simpler than $\mathcal{W}_{3}$. The
second team is simpler than $\mathcal{W}_{n-1}$. 
\end{proof}
Note that $\mathcal{W}_{2}$ is not a team but $\mathcal{L}(\mathcal{W}_{2})=\{ S_{3},C_{2},Z_{2}\}=\{ P_{3,1},P_{4,4},P_{4,5}\}$
is a winning team and so $\mathcal{W}_{2}$ is a winning set.

\begin{cor}
There is a winning team of size $s$ for all $s\in\mathbf{N}\setminus\{4\}$.
\end{cor}
\begin{proof}
The teams in Corollary~\ref{cor:Small_winners} are of size 1, 2
and 3. The team in Proposition~\ref{pro:The-full-family} has size
5. It is clear that $\mathcal{W}_{n}'=\mathcal{W}_{n}\cup\{ P_{5,6}\}$
is a team for $n\ge3$ (see Figure~\ref{cap:P_{5,6}.}). $\mathcal{W}_{n}'$
is a winner since it is simpler than $\mathcal{W}_{n}$. Since $|\mathcal{W}_{n}|=2n$
and $|\mathcal{W}_{n}'|=2n+1$, we have a winning team of size $s$
for all $s\ge6$.
\end{proof}
\begin{figure}
\begin{center}\includegraphics{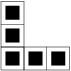}\end{center}

\caption{\label{cap:P_{5,6}.}$P_{5,6}$.}
\end{figure}

\section{Losing teams}

\begin{defn}
A $2$-\emph{paving} of the board is an irreflexive relation on the
set of cells where each cell is related to at most two other cells. 
\end{defn}
\begin{example}
Figure~\ref{cap:2-pavings.} visualizes some 2-pavings. Related cells
are connected by a tile. The dark cells show a fundamental set of
tiles. All the tiles are translations of the dark tiles by a linear
combination of the two given vectors with integer coefficients.%
\begin{figure}
\begin{tabular}{ccccc}
\includegraphics{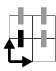}&
\includegraphics{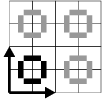}&
\includegraphics{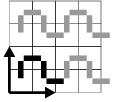}&
\includegraphics{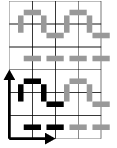}&
\includegraphics{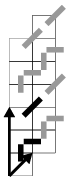}\tabularnewline
$\mathsf{P}_{A}$&
$\mathsf{P}_{B}$&
$\mathsf{P}_{C}$&
$\mathsf{P}_{D}$&
$\mathsf{P}_{E}$\tabularnewline
\end{tabular}

\caption{\label{cap:2-pavings.}2-pavings. Each picture shows four copies
of the fundamental set of tiles. }
\end{figure}
A 2-paving determines the following strategy for the breaker. In each
turn, the breaker marks the unmarked cells related to the cell last
marked by the maker. If there are fewer than two such cells then she
uses her remaining marks randomly. 
\end{example}
\begin{defn}
The strategy described above is called the \emph{paving strategy}
based on a 2-paving.
\end{defn}
\begin{prop}
\label{pro:If-the-breaker}If the breaker follows the paving strategy
then the maker cannot mark two related cells during a game.
\end{prop}
\begin{proof}
Suppose that it is the maker's turn and there is an empty cell $c$
related to the cell $d$ marked by the maker. But then cell $c$ was
empty after the maker marked cell $d$. So the breaker should have
been able to use one of her two marks on cell $c$ since cell $d$
is not related to more than two other cells. This is a contradiction.
\end{proof}
This result allows the breaker to win against certain sets of polyominoes.

\begin{defn}
If $\mathsf{P}$ is a $2$-paving such that every placement of the
polyomino $Q$ on the board contains a pair of related cells then
we say that $Q$ is \emph{killed} by $\mathsf{P}$. If every element
of a set $\mathcal{S}$ of polyominoes is killed by a 2-paving $\mathsf{P}$
then we say that $\mathcal{S}$ is killed by $\mathsf{P}$.
\end{defn}
Note that if $P\sqsubseteq Q$ and $P$ is killed by a $2$-paving,
then $Q$ is also killed by the same 2-paving. The following is an
easy consequence of Proposition~\ref{pro:If-the-breaker}.

\begin{prop}
A set of polyominoes killed by a $2$-paving is a losing set, the
breaker wins with the paving strategy.
\end{prop}
\begin{example}
Figure~\ref{cap:Polyominoes-Killers} shows the polyominoes up to
size $4$ with their killer 2-pavings. The table helps decide if a
team is a loser. For example $\{ S_{3},C_{2}\}$ is a loser because
it is killed by $\mathsf{P}_{C}$.

\begin{figure}
\begin{tabular}{|c|c|c|c|c|c|}
\hline 
&
$\mathsf{P}_{A}$&
$\mathsf{P}_{B}$&
$\mathsf{P}_{C}$&
$\mathsf{P}_{D}$&
$\mathsf{P}_{E}$\tabularnewline
\hline 
$S_{3}$&
&
$\bullet$&
$\bullet$&
&
\tabularnewline
\hline 
$L_{2}$&
$\bullet$&
&
&
&
\tabularnewline
\hline 
$S_{4}$&
&
$\bullet$&
$\bullet$&
$\bullet$&
$\bullet$\tabularnewline
\hline 
$L_{3}$&
$\bullet$&
$\bullet$&
$\bullet$&
&
$\bullet$\tabularnewline
\hline 
$T_{2}$&
$\bullet$&
$\bullet$&
$\bullet$&
$\bullet$&
\tabularnewline
\hline 
$C_{2}$&
$\bullet$&
&
$\bullet$&
$\bullet$&
$\bullet$\tabularnewline
\hline 
$Z_{2}$&
$\bullet$&
$\bullet$&
&
$\bullet$&
$\bullet$\tabularnewline
\hline
\end{tabular}

\caption{\label{cap:Polyominoes-Killers}Polyominoes and their killer 2-pavings.
$S_{1}$ and $S_{2}$ are not listed since those polyominoes are winners.}
\end{figure}

\end{example}
It is easy but tedious to check that a given 2-paving in fact kills
a polyomino. We used a computer program to verify our hand calculations. 

We used another computer program to find useful killer 2-pavings.
This program uses backtracking to pick more and more related cells
to find a 2-paving that kills a set of polyominoes on a finite region
of the board. The program places every polyomino inside the finite
region in every position that does not have a pair of related cells
yet. If one of these placements does not have two cells that can be
made related then the program backtracks. Otherwise the program picks
the placement that has the least number of cells that can be made
related and tries to consider every such pairing. The program stops
if the set cannot be killed by a 2-paving or if a killer 2-paving
is found. If a set cannot be killed by a 2-paving on a finite region
then of course it cannot be killed on the infinite board either. In
this case the set is called a \emph{paving winner}. The 2-pavings
found by the program are often chaotic at the boundary of the finite
region, but in most cases a pattern or sometimes several patterns
can be discovered in some portion of a sufficiently large region.

\begin{prop}
There is a losing team of size $s$ for all $s\in\mathbf{N}\cup\{\infty\}$.
\end{prop}
\begin{proof}
The teams $\{ C_{2},\ldots,C_{s+1}\}$ and $\{ C_{2},C_{3},\ldots\}$
are killed by $\mathsf{P}_{A}$.
\end{proof}
\begin{prop}
\label{pro:If-S-isnot}If $\mathcal{F}$ is a winning team then $S_{n}\in\mathcal{F}$
for some $n$.
\end{prop}
\begin{proof}
If $S_{n}$ is not in $\mathcal{F}$ for any $n$ then $\{ L_{2}\}\preceq\mathcal{F}$.
Hence $\mathcal{F}$ is a loser since $L_{2}$ is killed by $\mathsf{P}_{A}$.
\end{proof}
\begin{prop}
A set $\mathcal{S}$ containing polyominoes of size $5$ or larger
only is a loser.
\end{prop}
\begin{proof}
It is easy to see that $\mathcal{F}:=\{ S_{3},Z_{2}\}\preceq\mathcal{S}$
and $\mathcal{F}$ is killed by $\mathsf{P}_{B}$.
\end{proof}

\section{Classification of teams}

In this section we find all winning teams up to size 4. For each such
size $s$ we present a characterizing winning team $\mathcal{Y}_{s}$.
Then we show that a team $\mathcal{F}$ of size $s$ is a winner if
and only if it is simpler then $\mathcal{Y}_{s}$. To do this we use
a characterizing collection $\mathcal{N}_{s,1},\ldots,\mathcal{N}_{s,k_{s}}$
of losing teams and we show that if $\mathcal{F}$ is not simpler
than $\mathcal{Y}_{s}$ then there is a losing team in $\mathcal{N}_{s,i}$
that is simpler than $\mathcal{F}$. For size 4 teams we do not have
a characterizing winner since there are no size 4 winning teams. These
characterizing teams are shown in Figures~\ref{cap:Characterizing-1}--\ref{cap:Characterizing-4}.
Each $\mathcal{Y}_{i}$ is simpler than $\mathcal{W}$ of Corollary~\ref{cor:The-transfinite-family}
and so a winner. To show that the characterizing losing teams are
in fact losers, we provide killer 2-pavings in the figures.

\begin{prop}
$\mathcal{Y}_{1}=\{ S_{2}\}$, $\mathcal{N}_{1,1}=\{ S_{3}\}$ and
$\mathcal{N}_{1,2}=\{ L_{2}\}$ is a characterizing collection of
winners and losers for size $1$ teams.%
\begin{figure}
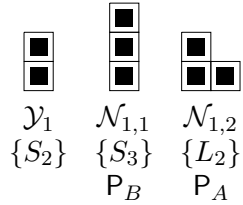

\begin{tabular}{ccc}
\includegraphics{2-1.eps}&
\includegraphics{3-1.eps}&
\includegraphics{3-2.eps}\tabularnewline
$\mathcal{Y}_{1}$&
$\mathcal{N}_{1,1}$&
$\mathcal{N}_{1,2}$\tabularnewline
$\{ S_{2}\}$&
$\{ S_{3}\}$&
$\{ L_{2}\}$\tabularnewline
&
$\mathsf{P}_{B}$&
$\mathsf{P}_{A}$\tabularnewline
\end{tabular}

\caption{\label{cap:Characterizing-1}Characterizing families for size 1.
Killer 2-pavings are listed for losing families.}
\end{figure}

\end{prop}
\begin{proof}
By \cite{sieben:wild}, the only size 1 winners are $\{ S_{1}\}$
and $\{ S_{2}\}$. Both of these are simpler than $\mathcal{Y}_{1}$.
Every other polyomino $P$ has at least 3 cells and so either $S_{3}$
or $L_{2}$ must be simpler then $P$. 
\end{proof}
\begin{prop}
$\mathcal{Y}_{2}=\{ S_{\infty},L_{2}\}$, $\mathcal{N}_{2,1}=\{ L_{2}\}$,
$\mathcal{N}_{2,2}=\{ S_{3},C_{2}\}$ and $\mathcal{N}_{2,3}=\{ S_{3},Z_{2}\}$
is a characterizing collection of winners and losers for size $2$
teams. %
\begin{figure}
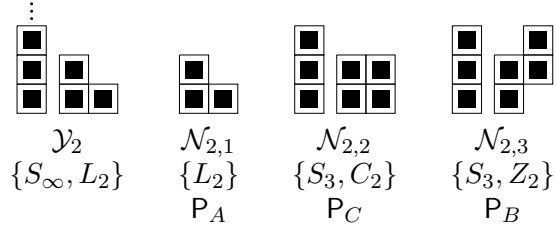

\begin{tabular}{ccccccc}
\includegraphics{omega-1.eps}~\includegraphics{3-2.eps}&
&
\includegraphics{3-2.eps}&
&
\includegraphics{3-1.eps}~\includegraphics{4-4.eps}&
&
\includegraphics{3-1.eps}~\includegraphics{4-5.eps}\tabularnewline
$\mathcal{Y}_{2}$&
&
$\mathcal{N}_{2,1}$&
&
$\mathcal{N}_{2,2}$&
&
$\mathcal{N}_{2,3}$\tabularnewline
$\{ S_{\infty},L_{2}\}$&
&
$\{ L_{2}\}$&
&
$\{ S_{3},C_{2}\}$&
&
$\{ S_{3},Z_{2}\}$\tabularnewline
&
&
$\mathsf{P}_{A}$&
&
$\mathsf{P}_{C}$&
&
$\mathsf{P}_{B}$\tabularnewline
\end{tabular}

\caption{\label{cap:Characterizing-2}Characterizing families for size 2.
Killer 2-pavings are listed for losing families.}
\end{figure}

\end{prop}
\begin{proof}
Let $\mathcal{F}$ be a team of size 2. If $S_{n}$ is not in $\mathcal{F}$
then $\mathcal{N}_{2,1}\preceq\mathcal{F}$ by the proof of Proposition~\ref{pro:If-S-isnot}.
So we can assume that $\mathcal{F}=\{ S_{n},Q\}$ for some $n\ge3$.
Note that if $n\le2$ then $\mathcal{F}$ cannot be a team. 

First assume that $|Q|\le4$. Then $Q\in\{ L_{2},L_{3},T_{2},C_{2},Z_{2}\}$
since $S_{i}$ is related to $S_{n}$. If $Q=L_{2}$ then $\mathcal{F}\preceq\mathcal{Y}$.
If $Q\in\{ L_{3},T_{2}\}$ then $\mathcal{N}_{2,2},\mathcal{N}_{2,3}\preceq\mathcal{F}$.
If $Q=C_{2}$ then $\mathcal{N}_{2,2}\preceq\mathcal{F}$. If $Q=Z_{2}$
then $\mathcal{N}_{2,3}\preceq\mathcal{F}$.

Next assume that $|Q|\ge5$. Then $Q$ is not skinny and so there
is an $R\in\{ L_{2},L_{3},T_{2},C_{2},Z_{2}\}$ such that $R\sqsubseteq Q$.
Hence $\{ S_{n},R\}\preceq\mathcal{F}$ and so $\mathcal{F}$ is characterized
since $\{ S_{n},R\}$ is characterized as we saw in the previous case.
\end{proof}
\begin{cor}
\label{cor:Size_2_winners}The only winning size $2$ teams are $\{ S_{\infty},L_{2}\}$
and $\{ S_{n},L_{2}\}$ for $n\ge3$.
\end{cor}
\begin{prop}
$\mathcal{Y}_{3}=\{ S_{3},C_{2},Z_{2}\}$, $\mathcal{N}_{3,1}=\{ L_{2}\}$,
$\mathcal{N}_{3,2}=\{ S_{3},Z_{2}\}$, $\mathcal{N}_{3,3}=\{ S_{3},C_{2},P_{5,10}\}$
and $\mathcal{N}_{3,4}=\{ S_{4},C_{2},Z_{2}\}$ is a characterizing
collection of winners and losers for size $3$ teams. %
\begin{figure}
\begin{tabular}{ccccccccc}
\includegraphics{3-1.eps}~\includegraphics{4-4.eps}~\includegraphics{4-5.eps}&
&
\includegraphics{3-2.eps}&
&
\includegraphics{3-1.eps}~\includegraphics{4-5.eps}&
&
\includegraphics{3-1.eps}~\includegraphics{4-4.eps}~\includegraphics{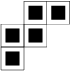}&
&
\includegraphics{4-1.eps}~\includegraphics{4-4.eps}~\includegraphics{4-5.eps}\tabularnewline
$\mathcal{Y}_{3}$&
&
$\mathcal{N}_{3,1}$&
&
$\mathcal{N}_{3,2}$&
&
$\mathcal{N}_{3,3}$&
&
$\mathcal{N}_{3,4}$\tabularnewline
$\{ S_{3},C_{2},Z_{2}\}$&
&
$\{ L_{2}\}$&
&
$\{ S_{3},Z_{2}\}$&
&
$\{ S_{3},C_{2},P_{5,10}\}$&
&
$\{ S_{4},C_{2},Z_{2}\}$\tabularnewline
&
&
$\mathsf{P}_{A}$&
&
$\mathsf{P}_{B}$&
&
$\mathsf{P}_{C}$&
&
$\mathsf{P}_{D}$\tabularnewline
\end{tabular}

\caption{\label{cap:Characterizing-3}Characterizing families for size 3.
Killer 2-pavings are listed for losing families.}
\end{figure}

\end{prop}
\begin{proof}
Let $\mathcal{F}$ be a team of size $3$. If $S_{n}$ is not in $\mathcal{F}$
then $\mathcal{N}_{3,1}\preceq\mathcal{F}$. So assume $\mathcal{F}=\{ S_{n},Q,R\}$
for some $n\ge3$. We do not have $L_{2}\in\mathcal{F}$ because every
polyomino is related to $S_{n}$ or $L_{2}$. Thus $|Q|,|R|\ge4$.

First consider the case when $|Q|=4=|R|$. Then $\{ Q,R\}\subseteq\{ L_{3},T_{2},C_{2},Z_{2}\}$.
If $\{ Q,R\}=\{ L_{3},T_{2}\}$ then $\mathcal{N}_{3,2}\preceq\{ S_{3}\}\preceq\mathcal{F}$.
If $Q\in\{ L_{3},T_{2}\}$ and $R=C_{2}$ then $\mathcal{N}_{3,3}\preceq\{ S_{3},C_{2}\}\preceq\mathcal{F}$.
If $Q\in\{ L_{3},T_{2}\}$ and $R=Z_{2}$ then $\mathcal{N}_{3,2}\preceq\mathcal{F}$.
If $\{ Q,R\}=\{ C_{2},Z_{2}\}$ then $n=3$ implies $\mathcal{F}=\mathcal{Y}_{3}$
and $n\ge4$ implies $\mathcal{N}_{3,4}\preceq\mathcal{F}$.

Next consider the case when $|Q|\ge4$ and $|R|\ge5$. Since $Q$
and $R$ are not skinny, there is an $\mathcal{S}\subseteq\{ P_{4,2},P_{4,3},P_{4,4},P_{4,5}\}$
with $|\mathcal{S}|\le2$ such that $\mathcal{S}\preceq\{ Q,R\}$.
Then $\mathcal{E}=\mathcal{L}(\{ S_{n}\}\cup\mathcal{S})\preceq\{ S_{n}\}\cup\mathcal{S}\preceq F$
and $1\le|\mathcal{E}|\le3$.

If $|\mathcal{E}|=1$ then $\mathcal{N}_{3,2}\preceq\mathcal{E}\preceq\mathcal{F}$.
If $|\mathcal{E}|=2$ then $\mathcal{E}$ is a loser by Corollary~\ref{cor:Size_2_winners},
since $\mathcal{E}$ has a polyomino with size $4$. Hence $\mathcal{N}_{2,1}$,
$\mathcal{N}_{2,2}$ or $\mathcal{N}_{2,3}$ is simpler than $\mathcal{E}$.
We have $\mathcal{N}_{3,1}=\mathcal{N}_{2,1}$, $\mathcal{N}_{3,3}\preceq\mathcal{N}_{2,2}$
and $\mathcal{N}_{3,2}=\mathcal{N}_{2,3}$ which implies $\mathcal{N}_{3,i}\preceq\mathcal{N}_{2,j}\preceq\mathcal{E}\preceq\mathcal{F}$
for some $i$ and $j$ as desired.%
\begin{figure}
\begin{tabular}{ccc}
\includegraphics{4-4.eps}&
&
\includegraphics{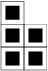}\tabularnewline
$C_{2}$&
$\sqsubseteq$&
$P_{5,4}$\tabularnewline
\end{tabular}$\qquad\qquad$\begin{tabular}{cccccc}
\includegraphics{4-5.eps}&
&
\includegraphics{5-4.eps}&
\includegraphics{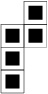}&
\includegraphics{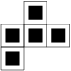}&
\includegraphics{5-10.eps}\tabularnewline
$Z_{2}$&
$\sqsubseteq$&
$P_{5,4}$~,&
$P_{5,8}$~,&
$P_{5,9}$~,&
$P_{5,10}$\tabularnewline
\end{tabular}

\caption{\label{cap:Size-5-descend}Descendants of $C_{2}$ and $Z_{2}$ with
size 5 .}
\end{figure}

Assume $|\mathcal{E}|=3$. If $\mathcal{E}\not=\mathcal{Y}_{3}$ then
$\mathcal{N}_{3,i}\preceq\mathcal{E}\preceq\mathcal{F}$ for some
$i$ by the first part of the proof. So it remains to consider the
case when $\mathcal{E}=\mathcal{Y}_{3}$. Then we must have an ancestor
$Q'$ of $Q$ and an ancestor $R'$ of $R$ such that $|Q'|=4$ and
$|R'|=5$. Figure~\ref{cap:Size-5-descend} shows the size 5 descendants
of $C_{2}$ and $Z_{2}$. From this we can see that either we have
$Q'=Z_{2}$ and $R'=P_{5,4}$ or we have $Q'=C_{2}$ and $R'\in\{ P_{5,4},P_{5,8},P_{5,9},P_{5,10}\}$.
In the first case $\mathcal{N}_{3,2}\preceq\{ S_{n},Z_{2},P_{5,4}\}\preceq\mathcal{F}$.
In the second case one of the following holds:\begin{align*}
\mathcal{N}_{3,3},\mathcal{N}_{3,4} & \preceq\{ S_{n},C_{2},P_{5,4}\}\preceq\mathcal{F}\\
\mathcal{N}_{3,4} & \preceq\{ S_{n},C_{2},P_{5,8}\}\preceq\mathcal{F}\qquad\text{($n\ge4$ since $S_{3}\sqsubseteq P_{5,8}$)}\\
\mathcal{N}_{3,4} & \preceq\{ S_{n},C_{2},P_{5,9}\}\preceq\mathcal{F}\text{\qquad($n\ge4$ since $S_{3}\sqsubseteq P_{5,9}$)}\\
\mathcal{N}_{3,3} & \preceq\{ S_{n},C_{2},P_{5,10}\}\preceq\mathcal{F}.\end{align*}
 
\end{proof}
We need a preliminary result before we can deal with size $4$ teams.
The polyominoes shown in Figure~\ref{cap:Squiggle-polyominoes.}
are called squiggle polyominoes.%
\begin{figure}
\begin{tabular}{ccccc}
\includegraphics{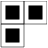}&
\includegraphics{4-5.eps}&
\includegraphics{5-10.eps}&
\includegraphics{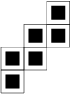}&
$\cdots$\tabularnewline
\end{tabular}

\caption{\label{cap:Squiggle-polyominoes.}Squiggle polyominoes.}
\end{figure}

\begin{prop}
\label{pro:no_size_3}A team $\mathcal{F}$ of size $4$ or more does
not have any polyominoes of size $3$ or less.
\end{prop}
\begin{proof}
It is clear that the full team $\mathcal{F}_{s}$ cannot be extended
to a larger team. Hence neither $S_{1}$ nor $S_{2}$ can be a member
of $\mathcal{F}$. We cannot have both $S_{3}$ and $L_{2}$ in $\mathcal{F}$
either. 

If $L_{2}\in\mathcal{F}$ then all the other polyominoes in $\mathcal{F}$
must be skinny since the non-skinny polyominoes are related to $L_{2}$.
Only one skinny polyomino is allowed so this limits the size of $\mathcal{F}$
to 2.

Suppose that $S_{3}\in\mathcal{F}$. The only polyominoes not related
to $S_{3}$ are $C_{2}$ and the squiggle polyominoes. Any two squiggle
polyominoes are related so $\mathcal{F}$ cannot contain more than
one. This limits the size of $\mathcal{F}$ to 3.
\end{proof}
\begin{prop}
There are no winning teams with size $4$. $\mathcal{N}_{4,1}=\{ L_{2}\}$,
$\mathcal{N}_{4,2}=\{ S_{2},Z_{2}\}$, $\mathcal{N}_{4,3}=\{ S_{2},C_{2},P_{5,10}\}$,
$\mathcal{N}_{4,4}=\{ S_{4},L_{3},C_{2},Z_{2}\}$ and $\mathcal{N}_{4,5}=\{ S_{4},T_{2},C_{2},Z_{2}\}$
is a characterizing collection of losers for size $4$ teams.%
\begin{figure}
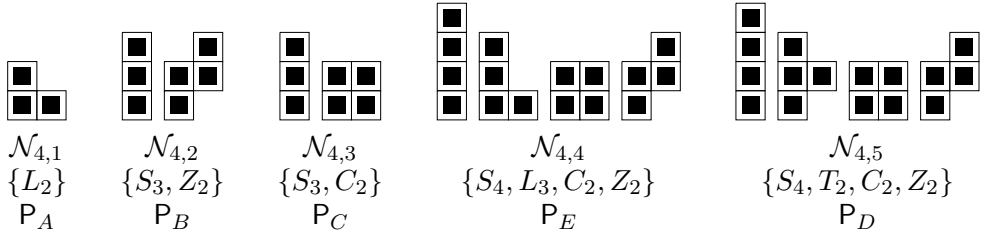

\begin{tabular}{ccccccccc}
\includegraphics{3-2.eps}&
&
\includegraphics{3-1.eps}~\includegraphics{4-5.eps}&
&
\includegraphics{3-1.eps}~\includegraphics{4-4.eps}&
&
\includegraphics{4-1.eps}~\includegraphics{4-2.eps}~\includegraphics{4-4.eps}~\includegraphics{4-5.eps}&
&
\includegraphics{4-1.eps}~\includegraphics{4-3.eps}~\includegraphics{4-4.eps}~\includegraphics{4-5.eps}\tabularnewline
$\mathcal{N}_{4,1}$&
&
$\mathcal{N}_{4,2}$&
&
$\mathcal{N}_{4,3}$&
&
$\mathcal{N}_{4,4}$&
&
$\mathcal{N}_{4,5}$\tabularnewline
$\{ L_{2}\}$&
&
$\{ S_{3},Z_{2}\}$&
&
$\{ S_{3},C_{2}\}$&
&
$\{ S_{4},L_{3},C_{2},Z_{2}\}$&
&
$\{ S_{4},T_{2},C_{2},Z_{2}\}$\tabularnewline
$\mathsf{P}_{A}$&
&
$\mathsf{P}_{B}$&
&
$\mathsf{P}_{C}$&
&
$\mathsf{P}_{E}$&
&
$\mathsf{P}_{D}$\tabularnewline
\end{tabular}

\caption{\label{cap:Characterizing-4}Characterizing teams for size 4. Killer
2-pavings are listed for losing teams. No winning team is required.}
\end{figure}

\end{prop}
\begin{proof}
Let $\mathcal{F}$ be a team of size $4$. If $S_{n}$ is not in $\mathcal{F}$
then $\mathcal{N}_{3,1}\preceq\mathcal{F}$. So assume $\mathcal{F}=\{ S_{n},P,Q,R\}$.
for some $n\ge3$. By Proposition~\ref{pro:no_size_3} we can assume
that $n,|P|,|Q|,|R|\ge4$. There is an $\mathcal{S}\subseteq\{ P_{4,2},\ldots,P_{4,5}\}$
with $|\mathcal{S}|\le3$ such that $\mathcal{S}\preceq\{ P,Q,R\}$.
Then $\mathcal{E}=\mathcal{L}(\{ S_{n}\}\cup\mathcal{S})\preceq\{ S_{n}\}\cup\mathcal{S}\preceq\mathcal{F}$
and $1\le|\mathcal{E}|\le4$. 

If $|\mathcal{E}|=1$ then $\mathcal{N}_{4,2},\mathcal{N}_{4,3}\preceq\mathcal{E}\preceq\mathcal{F}$.
If $|\mathcal{E}|=2$ then one of the following holds:\begin{align*}
\mathcal{N}_{4,2},\mathcal{N}_{4,3},\mathcal{N}_{4,4} & \preceq\{ S_{n},L_{3}\}=\mathcal{E\preceq\mathcal{F}}\\
\mathcal{N}_{4,2},\mathcal{N}_{4,3},\mathcal{N}_{4,5} & \preceq\{ S_{n},T_{2}\}=\mathcal{E\preceq\mathcal{F}}\\
\mathcal{N}_{4,3},\mathcal{N}_{4,4},\mathcal{N}_{4,5} & \preceq\{ S_{n},C_{2}\}=\mathcal{E\preceq\mathcal{F}}\\
\mathcal{N}_{4,2},\mathcal{N}_{4,4},\mathcal{N}_{4,5} & \preceq\{ S_{n},Z_{2}\}=\mathcal{E\preceq\mathcal{F}}.\end{align*}
If $|\mathcal{E}|=3$ then one of the following\begin{align*}
\mathcal{N}_{4,2},\mathcal{N}_{4,3} & \preceq\{ S_{n},L_{3,}P_{4,3}\}=\mathcal{E\preceq\mathcal{F}}\\
\mathcal{N}_{4,3},\mathcal{N}_{4,4} & \preceq\{ S_{n},L_{3},C_{2}\}=\mathcal{E\preceq\mathcal{F}}\\
\mathcal{N}_{4,2},\mathcal{N}_{4,4} & \preceq\{ S_{n},L_{3},Z_{2}\}=\mathcal{E\preceq F}\\
\mathcal{N}_{4,3},\mathcal{N}_{4,5} & \preceq\{ S_{n},T_{2},C_{2}\}=\mathcal{E\preceq\mathcal{F}}\\
\mathcal{N}_{4,2},\mathcal{N}_{4,5} & \preceq\{ S_{n},T_{2},Z_{2}\}=\mathcal{E\preceq\mathcal{F}}\\
\mathcal{N}_{4,4},\mathcal{N}_{4,5} & \preceq\{ S_{n},C_{2},Z_{2}\}=\mathcal{E\preceq\mathcal{F}}\end{align*}
holds. Finally if $|\mathcal{E}|=4$ then one of the following holds:\begin{align*}
\mathcal{N}_{4,3} & \preceq\{ S_{n},L_{3},T_{2},C_{2}\}=\mathcal{E\preceq\mathcal{F}}\\
\mathcal{N}_{4,2} & \preceq\{ S_{n},L_{3},T_{2},Z_{2}\}=\mathcal{E\preceq\mathcal{F}}\\
\mathcal{N}_{4,4} & \preceq\{ S_{n},L_{3},C_{2},Z_{2}\}=\mathcal{E\preceq\mathcal{F}}\\
\mathcal{N}_{4,5} & \preceq\{ S_{n},T_{2},C_{2},Z_{2}\}=\mathcal{E\preceq\mathcal{F}}.\end{align*}

\end{proof}
\begin{defn}
A team $\mathcal{Y}$ of polyominoes is called an $n$-\emph{super
winner} if each winning team with size at most $n$ is simpler than
$\mathcal{Y}$.
\end{defn}
\begin{example}
$\mathcal{Y}_{s}$ is an $s$-super winner for $s\in\{1,2\}$. $\mathcal{W}$
in Corollary~\ref{cor:The-transfinite-family} is a $4$-super winner.
\end{example}
The main result of our paper is the following.

\begin{thm}
A team of polyominoes containing fewer than $5$ polyominoes is a
winner if and only if it is simpler than $\mathcal{W}$.
\end{thm}

\section{Further questions}

There are several questions to be answered about set games. 

\begin{enumerate}
\item The teams $\mathcal{Y}_{2}$ and $\mathcal{W}$ are infinite winners.
Both of these are unbounded. Is there an infinite winning team that
is bounded?
\item Even though there are no winning teams with size 4, we could say that
$\mathcal{Y}_{4}=\{ S_{1}\}$ is a characterizing winner for size
4 teams. So there is a characterizing winning team for sizes from
1 to 4. Is there a characterizing winner for each size?
\item Is there an $s$-super winner for each $s$? Is there a super winner
that is $s$-super for each $s$?
\item Is there a useful notion of a super loser?
\item Are there any characterizing or super winners in the unbiased or differently
biased set games played on triangular, hexagonal and higher dimensional
rectangular boards?
\end{enumerate}
\begin{center}\bibliographystyle{amsplain}
\bibliography{game}

\providecommand{\bysame}{\leavevmode\hbox to3em{\hrulefill}\thinspace}
\providecommand{\MR}{\relax\ifhmode\unskip\space\fi MR }
\providecommand{\MRhref}[2]{%
  \href{http://www.ams.org/mathscinet-getitem?mr=#1}{#2}
}
\providecommand{\href}[2]{#2}
\begin{thebibliography}{10}

\bibitem{beck:remarks}
J{\'o}zsef Beck, \emph{Remarks on positional games. {I}}, Acta Math. Acad. Sci.
  Hungar. \textbf{40} (1982), no.~1-2, 65--71.

\bibitem{beck:biased}
\bysame, \emph{Biased {R}amsey type games}, Studia Sci. Math. Hungar.
  \textbf{18} (1983), no.~2-4, 287--292.

\bibitem{bode.harborth:hexagonal}
Jens-P. Bode and Heiko Harborth, \emph{Hexagonal polyomino achievement},
  Discrete Math. \textbf{212} (2000), no.~1-2, 5--18.

\bibitem{bode.harborth:triangle}
\bysame, \emph{Triangle polyomino set achievement}, Congr. Numer. \textbf{148}
  (2001), 97--101.

\bibitem{erdos.selfridge:on}
P.~Erd{\H{o}}s and J.~L. Selfridge, \emph{On a combinatorial game}, J.
  Combinatorial Theory Ser. A \textbf{14} (1973), 298--301.

\bibitem{gardner.martin:mathematical}
Martin Gardner, \emph{Mathematical games}, Sci. Amer. \textbf{240} (1979),
  18--26.

\bibitem{harary:achieving}
Frank Harary, \emph{Achieving the skinny animal}, Eureka \textbf{42} (1982),
  8--14.

\bibitem{harary:is}
\bysame, \emph{Is {S}naky a winner?}, Geombinatorics \textbf{2} (1993), no.~4,
  79--82.

\bibitem{harary.harborth.ea:handicap}
Frank Harary, Heiko Harborth, and Markus Seemann, \emph{Handicap achievement
  for polyominoes}, Congr. Numer. \textbf{145} (2000), 65--80.

\bibitem{sieben:snaky}
N{\'a}ndor Sieben, \emph{Snaky is a {$41$}-dimensional winner}, Integers
  \textbf{4} (2004), G5, 6 pp. (electronic).

\bibitem{sieben:wild}
\bysame, \emph{Wild polyomino weak {$(1,2)$}-achievement games.},
  Geombinatorics \textbf{13(4)} (2004), 180--185.

\bibitem{sieben:polyominoes}
\bysame, \emph{Polyominoes with minimum site-perimeter and full set achievement
  games}, European J. Combin. \textbf{29} (2008), no.~1, 108--117.

\bibitem{sieben.deabay:polyomino}
N{\'a}ndor Sieben and Elaina Deabay, \emph{Polyomino weak achievement games on
  {$3$}-dimensional rectangular boards}, Discrete Mathematics \textbf{290}
  (2005), 61--78.

\end{thebibliography}
\end{center}

\end{document}